\documentclass[11pt]{amsart}
\usepackage{graphicx}
\usepackage{latexsym}
\usepackage{amsfonts,amsmath,amssymb}
\newtheorem{theorem}{Theorem}[section]																
\newtheorem{lemma}[theorem]{Lemma}

\usepackage{color}

\def\bi{\par\bigskip\noindent}

\def\la{\lambda}

\def\la{\lambda}
\def\part{\partial}

\def\b1{\bold 1}

\newcommand{\beq}{\begin{equation}}
\newcommand{\eeq}{\end{equation}}

\newcommand{\id}{\mathrm{id}}

\theoremstyle{remark}

\numberwithin{equation}{section}
\date{\today}

\begin{document}

\title[Product of two cycles]{Counting pairs of cycles whose product is a permutation with restricted cycle lengths}
\author {Mikl\'os B\'ona and Boris Pittel}
\address{Department of Mathematics, University of Florida, Gainesville, FL 32611}
\email{bona@ufl.edu}
\address{Department of Mathematics, Ohio State, University, Columbus, OH 43210}
\email{pittel.1@osu.edu}

\date{}
\begin{abstract} 
We find exact and asymptotic formulas for the number of pairs $(s,z)$ of $N$-cycles $s$ and $z$ such that all cycles of the product $s\cdot z$ have lengths from a given set of integers. We then apply one of these results to prove a surprisingly high
lower bound for the number of permutations whose block transposition distance from the identity is at least
$(N+1)/2$. 
\end{abstract}


\maketitle

\section{Introduction}
\subsection{Permutations with restricted cycle lengths}

We are interested in counting {\it some\/} of the permutations of a set $[N]$ whose cycle lengths belong to a given subset $A\subseteq \{1,2,\dots\}$. Permutations of this kind were studied by Sachkov \cite{Sach},
under the name of $A$-permutations. We stress "some", since our focus is on permutations of $[N]$ that are products of pairs of $N$-long cycles. Some special cases have direct consequences
in the theory of biologically motivated sorting algorithms.  We will explore a very surprising one of these, 
related to sorting a permutation by block transpositions, in Section \ref{btsorting}.  One of our permutation-counting 
results will imply the existence of a high number of permutations of length $N$ that take at least $(N+1)/2$
block transpositions to sort. This is remarkable, because for most values of $N$, there are no known examples
of permutations of length $N$ that would take more than that many block transpositions to sort.

Let $\Bbb P_N(A)$ denote the total number of {\it all\/} such 
permutations. It is well known (cf. B\'ona \cite{Bona}) that
\begin{equation}\label{1*}
\Bbb P_N(A)=N! \cdot [x^N] \exp(F_A(x)),\quad F_A(x):=\sum_{r\in A}\tfrac{x^r}{r}.
\end{equation}
Equivalently, $p_N(A)$, the probability that the uniformly random permutation of $[N]$ has all its cycle lengths from $A$,
is given by
\begin{equation}\label{2*}
p_N(A)=[x^N] \exp(F_A(x)).
\end{equation}
We will show that this leads to the explicit formulas
\begin{equation}\label{new?}
\begin{aligned}
p_k(A)&=\sum_{\mu_r\ge 0,\,r\in A^c\atop
\sum_r r\mu_r\le k}\prod_{r\in A^c}\tfrac{(-1)^{\mu_r}}{r^{\mu_r}\mu_r!},\\
p_{\ell}(A^c)&=\sum_{\sum_{r\in A^c}r \mu_r=\ell}\,\prod_{r\in A^c}\tfrac{1}{r^{\mu_r}\mu_r!},
\end{aligned}
\end{equation}
where  $A^c:=\{1,2,\dots\}\setminus A$ is the complement of $A$ in the set of all positive integers.
In the context of sorting algorithms, particularly relevant are the cases $A=\mathcal E$, the set of even numbers, and $A=\mathcal O$, the set of odd numbers.

The task of counting {\em all} such permutations is well resolved. It is well known \cite{Bona,stanleyec2} that
\begin{equation}\label{2***}
F_{\mathcal E}(x)=\log (1-x^2)^{-1/2},\quad F_{\mathcal O}(x)=\log\bigl(\tfrac{1+x}{1-x}\bigr)^{1/2}.
\end{equation}
It follows that
\begin{equation}\label{2**}
\begin{aligned}
p_N(\mathcal E)&=2^{-N}\binom{N}{N/2},
\quad N\text{ even},\\
p_N(\mathcal O)&=\left\{\begin{aligned}
&2^{-N}\binom{N}{N/2},&&N\text{ is even},\\
&2^{-N+1}\binom{N-1}{(N-1)/2},&&N\text{ is odd}.\end{aligned}\right.
\end{aligned}
\end{equation}

We point out that the enumeration formulas that follow from \eqref{2***} and \eqref{2**} have known combinatorial proofs. The interested reader should consult the references given in the solution Exercise 5.10 in \cite{stanleyec2}. 

It is significantly more challenging to obtain a tractable counterpart of \eqref{2*} for $q_N(A)$, {\em the probability that the {\it product\/} of two independent, uniformly random  $N$-long cycles has all its cycle lengths from a given set} $A$. Our main result consists of two complementary parts. First, for an arbitrary $A\subset \{1,2,\dots\}$, we will 
get a sum-type formula for the numbers $q_N(A)$ in terms of the numbers $p_k(A)$ and  $p_{\ell}(A^c)$, including a wide range extension of well known results for $A=\{2,3,\dots \}$. While explicit, the expression does not lend itself easily to asymptotic analysis, since the summands have alternating signs. For two important cases, $A$ being either the set of even numbers or the set of odd numbers, we find a way to obtain the sum-type formulas with positive summands only, 
much more conducive to asymptotic computations. We will present these formulas and their proofs next.

\subsection{Products of Maximal Cycles}
The following theorem summarizes  the formulas we will obtain for $q_N(A)$.

\begin{theorem}\label{thm1*}
For all positive integers $N$, the following hold. 
\begin{enumerate} 
\item[({\bf a})] For a given $A\subset\{1,2,\dots\}$, 

\begin{multline}\label{4*}
q_N(A)=\tfrac{N}{N+1}
\biggl(\sum_{k+\ell=N \atop 0\leq k,\ell \leq N}(-1)^{\ell}\binom{N}{\ell}^{-1} p_k(A)\bigl(p_{\ell}(A^c)-p_{\ell-1}(A^c)\bigr)\biggr),\\
\end{multline}
where $p_0(A)=p_{0}(A^c)=1$, $p_1(A^c):=0$; for $k,\ell\ge 1$, $p_k(A),\,p_{\ell}(A^c)$ are given by \eqref{new?}. 
\item[({\bf b})]  For $A=\mathcal E$ (even positive integers) and $A=\mathcal O$ (odd positive integers), we have
\begin{equation}\label{5*}
\begin{aligned}
q_N(\mathcal E)&=\tfrac{2N}{(N+1) 2^N}\sum_{k+\ell=N/2,  \atop 0\leq k,\ell \leq N}\tfrac{1}{k}\binom{2(k-1)}{k-1}\binom{2\ell}{\ell}\biggl(1-\binom{N}{2\ell}^{-1}\biggr)\!,\,N\text{ even},\\
q_N(\mathcal O)&=\tfrac{N}{2^N(N+1)}\sum_{2k+2\ell=N,  \atop 0\leq k,\ell \leq N}\binom{2k}{k}\binom{2\ell}{\ell}\binom{N}{2\ell}^{-1}\\
&\quad+\tfrac{N}{2^{N-2}(N+1)}\sum_{2k+2\ell=N-2,  \atop 0\leq k,\ell \leq N}\binom{2k}{k}\binom{2\ell}{\ell}
\binom{N}{2\ell+1}^{-1},\,N\text{ even},\\
q_N(\mathcal O)&=\tfrac{N}{(N+1)2^{N-1}}\!\!\sum_{2\ell+2k=N-1,  \atop 0\leq k,\ell \leq N}\!\!\binom{2\ell}{\ell}\binom{2k}{k}\biggl(\!\binom{N}{2\ell}^{-1}\!\! +\binom{N}{2\ell+1}^{-1}\biggr),\,N\text{ odd}.
\end{aligned}
\end{equation}
\end{enumerate}
 In each of the above sums the dominant terms are those with the smallest {\it admissible\/} $k$.
Consequently $q_N(\mathcal E),\,q_N(\mathcal O)\sim (\pi N/2)^{-1/2}$, as $N\to\infty$.
\end{theorem}

The set $\mathcal E$ is a special case of $\mathcal D$, the full set of positive integers divisible by a fixed $d>1$. So, we count the permutations whose cycle lengths are all divisible by $d$. Here  
\begin{equation}\label{6*}
F_{\mathcal D}=\tfrac{1}{d} \log(1-x^d)^{-1},
\end{equation}
and consequently, for $N\equiv 0(\text{mod }d)$, 
\begin{equation*}
p_N(\mathcal D)=\tfrac{1}{(N/d)!}\prod_{r=0}^{N/d-1}(r+1/d).
\end{equation*}
Similarly to  the proof for $q_N(\mathcal E)$, we will show in the Appendix 1 that
\begin{equation}\label{7*}
\begin{aligned}
q_N(\mathcal D)&=\tfrac{N}{d(N+1)}\sum_{dk+d\ell=N,  \atop 0\leq k,\ell \leq N}\tfrac{1}{k!\ell!}\biggl(1-(-1)^{dk}\binom{N}{d\ell}^{-1}\biggr)\\
&\qquad\times \prod_{r=1}^{k-1}(r-1/d)\prod_{r'=0}^{\ell-1}(r'+1/d),
\end{aligned}
\end{equation}
and for $d=2$ the equation \eqref{7*} yields the top  formula in \eqref{5*}. 

 Note. For $A=\{2,3,\dots\}$, a permutation $p$ of $[N]$ is type-A if and only if $p$ has no cycles of length $1$, that is, when  $p$ is a derangement. In this case, the formula \eqref{new?} becomes
\begin{equation}\label{star}
p_k(A)=\sum_{\mu=0}^k \tfrac{(-1)^{\mu}}{\mu!},\quad p_{\ell}(A^c)=\tfrac{1}{\ell!}.
\end{equation}
With some work (see Appendix 2), the equations  \eqref{4*} and \eqref{star} imply
\begin{equation}\label{new10*}
(N-1)! \,q_N(A)=(-1)^{N-1}\sum_{j=0}^{N-2}\tfrac{(-1)^j}{j!}\sum_{j\le k<N}(-1)^k k! .
\end{equation}
The reason we are interested in $(N-1)! \,q_N(A)$ is that it equals the number of $N$-long cycles $p$ such that
the product $(12\dots N)p$ is a derangement. (Comfortingly, the RHS of \eqref{new10*} is indeed an integer for all $N$, since $k!/j!$ is an integer for $k\ge j$.) And it is well-known that this property holds if and only there is no $i\in [N]$ such that $p(i)=i-1(\text{mod }N)$, see for instance Stanley \cite{Sta1966}, or Charalambides \cite{Cha}. It follows from the inclusion-exclusion principle that $C_N$, the total number of such maximal cycles $p$ is given by
\begin{equation}\label{new11*}
C_N=N!\sum_{\mu=0}^{N-1}\tfrac{(-1)^{\mu}}{(N-\mu)\mu!}+(-1)^{N}.
\end{equation}
Consequently, $(N-1)! \,q_N(A)=C_N$, which is not obvious at all. Using Maple, we checked that this identity holds for all $N\le 10$. At this moment it is unclear to us how
 to prove this identity formally, that is, without using the fact that both sides count the same objects.

\section{Proof of Theorem \ref{thm1*}.} We start with the following general formula, Pittel \cite{Pit1} and Zagier \cite{Zag}. Let $S_N$ be the set of all permutations on $[N]$, and let $C_1,\dots, C_k$ be some conjugacy classes from $S_N$. Let $\sigma_1,\dots, \sigma_k$ be the random {\it independent\/} elements of $S_N$, where $\sigma_j$ is distributed {\it uniformly\/} on $C_j$. For
$\sigma:=\sigma_1\cdots\sigma_k$, we have
\begin{equation}\label{1}
\Bbb P(\sigma=s)=P_{\sigma}(s):= \tfrac{1}{N!}\sum_{\la}(f^{\la})^{-k+1}\chi^{\la}(s)\prod_{j=1}^k\chi^{\la}(C_j).
\end{equation}
Here the sum is taken over all integer partitions (Young diagrams) of $N$, and $f^{\la}$ is the dimension of the irreducible representation of $S_N$ associated with $\la$, given by the hook formula
\begin{equation}\label{2}
f(\la)=\tfrac{N!}{\prod_{(i,j)\in \la}h_{i,j}}.
\end{equation}
Here $h_{i,j}$ is the cardinality of the hook for a cell $(i,j)$, that is,  $1$ plus the number of cells in $\la$ to the east and  to the south from $(i,j)$. Furthermore, $\chi^{\la}(s)$ and $\chi^{\la}(C_j)$ are the values of the associated character $\chi^{\la}$ at $s$ and at the elements of the conjugacy classes $C_j$.

Let $\mathcal C_{\bold p}$ be a generic conjugacy class consisting of all permutations of $[N]$, with cycle lengths $p_1\ge p_2\ge\cdots$,  with $\sum_jp_j=N$.  By the Murnaghan-Nakayama rule, (see Macdonald \cite{Mac},  Sagan \cite{Sag}. Stanley \cite{stanley}), the common value of  $\chi^{\la}(s)$ for $s\in \mathcal C_p$ is given by
\begin{equation}\label{3}
\chi^{\la}(\mathcal C_{\bold p})=\sum_T(-1)^{\text{ht}(T)}.
\end{equation}

 Here the sum is over all {\it rim hook} diagrams $T$ of shape $\la$ and type $\bold p$, that is, over all ways to empty the diagram $\la$ by successive deletions of {\it border strips}, (strips  consisting of only current border cells), one strip at a time, of {\it weakly increasing\/} lengths $p_j$. Further, $\text{ht}(T)$ is the sum of heights of the individual border strips (number of their rows minus $1$) in $T$. We will see that, contrary to its forbidding appearance, the formula \eqref{3} {\it alone} leads to a surprisingly simple expression for the values $\chi^{\la}(\mathcal C)$, if $\mathcal C$ is the conjugacy classs of one-cycle permutations.

We focus on $k=2$, and $C_1=C_2=\mathcal C_N$, the set of all maximal, $N$-long, cycles.  Here the composition $\bold p$ consists of a single component $N$, and the range of $T$ is empty unless the diagram $\la$ is a single hook $\la^*$, with one horizontal row of length $\la_1$ and one column of height $\la^1$, $\la_1+\la^{1}=N+1$, in which case $\chi^{\la*}(\mathcal C_N)=(-1)^{\la^1-1}$. And by the hook formula \eqref{2}, we obtain
 
\begin{equation}\label{4}
f(\la^*)=\tfrac{N!}{N\prod\limits_{i=1}^{\la^1-1}\!\!i\,\,\,\times\prod\limits_{j=1}^{\la_1-1}\! j}=\binom{N-1}{\la_1-1}.
\end{equation}

Furthermore, by \eqref{3}, given one-hook diagram $\la^*$, the value of $\chi^{\la^*}(s)$ depends on a generic permutation $s$ only through $\boldsymbol\nu=\boldsymbol\nu (s)$, where $\nu_r=\nu_r(s)$ is the total number of $r$-long cycles in $s$. We need the following tool. 
\begin{lemma}\label{lem*}  Let $\mathcal C_{\bold p}$ be a generic conjugacy class consisting of all permutations of $[N]$, with cycle lengths $p_1\ge p_2\ge\cdots$,  with $\sum_jp_j=N$.
Let $\nu_j$ be the number of $p_i$'s equal to $j$. Then
\begin{equation}\label{5}
\chi^{\la^*}(\mathcal C_{\bold p})=(-1)^{\la^1}[x^{\la_1}]\tfrac{x}{1-x}\prod_{r\ge 1}(x^r-1)^{\nu_r},
\end{equation}
where $\la_1$ and $\la^1$ are the arm length  and the leg  length of the hook $\la^*$.
\end{lemma}

\begin{proof} The argument is an  edited proof of this identity in \cite{Pit1}. In a generic rim hook diagram $T$,  the last rim hook  is a hook diagram with a row and a column of sizes $\mu_1$ and $\mu^1$, with $\mu_1\le \la_1$  and $\mu^1\le \la^1$. Define $\ell$ by $\ell+1=\mu_1+\mu^1$. All the preceding rim hooks in $T$ are horizontal and vertical segments of the 
arm and of the leg of $\la^*$. Let us evaluate the total number  of the rim hooks diagrams with $h_r$, the number of the horizontal hooks of size $r$; so $h_r\in [0,\nu_r]$ for $r\neq \ell$ and $h_{\ell}\in [0,\nu_{\ell}-1]$. The admissible $\bold h:=\{h_r\}$ must meet the additional constraint 
\begin{equation}\label{5.1}
\sum_r h_r r+\mu_1=\la_1. 
\end{equation}
If $v_{r}$ is similarly defined as the total number of vertical hooks of size $r$, then the preceding constraint implies that 
$\sum_r v_r r+\mu^1=\la^1$. The total number
of the rim hook diagrams $T$, with parameters $\mu_1$,  $\mu^1$, and $ \bold h$, is
\[
\binom{\nu_{\ell}-1}{h_{\ell}}\prod_{r\neq \ell}\binom{\nu_r}{h_r}=\prod_r\binom{\nu_r-\delta_{\ell,r}}{h_r},\quad \delta_{\ell,r}=\delta_{\ell,r}:=\Bbb I(r=\ell).
\]
We do not have to multiply the above product by the multinomial coefficient $\binom{\nu-1}{\nu_1,\dots,\nu_{\ell-1}, \nu_{\ell}-1,\nu_{\ell+1},\dots}$ since the rim hooks are deleted in the fixed, increasing order of their lengths.  Further, the $\mu^1$-long leg of the last (longest) hook contributes $\mu^1-1$ to the height $T$, while the total contribution
of the vertical rim hooks is $(\la^1-\mu^1)$, the sum of their sizes, minus $\sum_r(\nu_r-\delta_{\ell,r}-h_r)$, their total number. So
\begin{align*}
ht(T)&=(\mu^1-1)+(\la^1-\mu^1)-\sum_r(\nu_r-\delta_{\ell,r}-h_r)\\
&\equiv \la^1-1+\sum_r(\nu_r-\delta_{\ell,r}-h_r)\, \, (\text{mod }2).
\end{align*}
Therefore the total contribution to $\chi^{\la^*}(s)$ from the rim hook diagrams with the last rim hook $\mu^*$ is
\begin{align*}
&(-1)^{\la^1-1}\sum_{\bold h\text{ meets }\eqref{5.1}}\prod_r (-1)^{\nu_r-\delta_{\ell,r}-h_r}\binom{\nu_r-\delta_{\ell,r}}{h_r}\\
&=\,\, (-1)^{\la^1-1}[\xi^{\la_1-\mu_1}]\prod_r\sum_{h_r}(-1)^{\nu_r-\delta_{\ell,r}-h_r}(\xi^r)^{h_r}\binom{\nu_r-\delta_{\ell,r}}{h_r}\\
&=\,\, (-1)^{\la^1-1}[\xi^{\la_1-\mu_1}]\prod_r (\xi^r-1)^{\nu_r-\delta_{\ell,r}}.
\end{align*}
Since $\la_1-\mu_1\in [\la_1 -\min(\ell,\la_1), \la_1-1]$, we have
\begin{align*}
\chi^{\la^*}(s)&=(-1)^{\la^1-1}\sum_{t=\la_1-\min(\ell,\la_1)}^{\la_1-1}[\xi^t]\prod_r(\xi^r-1)^{\nu_r-\delta_{\ell,r}}\\
&=(-1)^{\la^1-1}[\xi^{\la_1-1}]\sum_{\ell}\biggl(\,\sum_{u=0}^{\min(\ell,\la_1)-1}\xi^u\biggr)\prod_r(\xi^r-1)^{\nu_r-\delta_{\ell,r}}\\
&=(-1)^{\la^1-1}[\xi^{\la_1-1}]\,\sum_{\ell}\tfrac{\xi^{\min(\ell,\la_1)}-1}{(\xi-1)(\xi^{\ell}-1)}\prod_r(\xi^r-1)^{\nu_r}.\\
\end{align*}
Here
\[
\tfrac{\xi^{\min(\ell,\la_1)}-1}{(\xi-1)(\xi^{\ell}-1)}=\left\{\begin{aligned}
&\tfrac{1}{\xi-1},&&\ell\le \la_1,\\
&\tfrac{\xi^{\la_1}-1}{(\xi-1)(\xi^{\ell}-1)},&&\ell>\la_1.\end{aligned}\right.
\]
Since
\[
\tfrac{\xi^{\la_1}-1}{(\xi-1)(\xi^{\ell}-1)}=\tfrac{1}{\xi-1}+\sum_{j\ge \la_1}c_j \xi^j,
\]
we see that
\[
\chi^{\la^*}(s)=(-1)^{\la^1-1}[\xi^{\la_1-1}] \tfrac{1}{\xi-1}\prod_r(\xi^r-1)^{\nu_r}.
\]

\end{proof}
Substituting formulas \eqref{4} and \eqref{5} into \eqref{1} and setting $\nu=\sum_r\nu_r$, we obtain the equality
\begin{equation}\label{6}
\Bbb P(\sigma=s)=\tfrac{1}{N!}\sum_{\la_1\in [N]}\binom{N-1}{\la_1-1}^{-1}\!\!(-1)^{\la^1} [x^{\la_1}]\tfrac{x}{1-x}\prod_r (x^r-1)^{\nu_r}.
\end{equation}

Furthermore, using the terminology in Sachkov \cite{Sach}, for a given  $A\subseteq \{1,2,\dots\}$, we call $s$ an $A$-permutation if $\nu_r=0$ for all $r\notin A$. Let us use \eqref{6} to obtain a hopefully tractable formula for $q_N(A)$,  the probability that the product $s$ of two $N$-cycles is an $A$-permutation.
To this end, we need to sum the right-hand side of \eqref{6} over all $\boldsymbol \nu$ with $\nu_r=0$ for $r\notin A$. 
To begin, notice that the RHS of \eqref{6} depends on the cycle-type $\{\nu_r(s)\}$ of the permutation $s$ only via the rightmost
product $\prod_r(x^r-1)^{\nu_r}$. Observing that 
the total number of permutations with given cycle counts $\boldsymbol\nu=\{\nu_r\}$ is
\[
\tfrac{N!}{\prod\limits_r (r!)^{\nu_r}\nu_r!}\cdot\prod_{r'}((r'-1)!)^{\nu_{r'}}=\tfrac{N!}{\prod\limits_r r^{\nu_r}\nu_r!},
\]
we have
\begin{multline*}
\sum_{\boldsymbol\nu:\,\nu_r=0,\,\forall\,r\notin A\atop \sum_r r\nu_r=N}\tfrac{N!}{\prod\limits_r r^{\nu_r}
\nu_r!}\prod_{r} (x^{r}-1)^{\nu_{r}}=N!\sum_{\boldsymbol\nu:\,\nu_r=0,\,\forall\,r\notin A\atop \sum_r r\nu_r=N}
\prod_r\bigl(\tfrac{x^r-1}{r}\bigr)^{\nu_r}\tfrac{1}{\nu_r!}\\
=N!\, [y^N]\prod_{r\in A}\sum_{m\ge 0}\bigl(\tfrac{x^r-1}{r}\bigr)^m \cdot \tfrac{y^{mr}}{m!}
=N!\, [y^N]\prod_{r\in A}\exp\Bigl(\tfrac{y^r(x^r-1)}{r}\Bigr)\\
=N!\, [y^N]\exp\bigl(F(xy)-F(y)\bigr),\quad F(\eta):=\sum_{r\in A}\tfrac{\eta^r}{r}.
\end{multline*}

Consequently,
 \begin{multline}\label{7}
q_N(A)=\sum_{s\text{ is type}-A} \Bbb P(\sigma=s)\\
=\tfrac{1}{N!}\sum_{\la_1\in [N]}\binom{N-1}{\la_1-1}^{-1}(-1)^{\la^1}[x^{\la_1}]\tfrac{x}{1-x}
\sum_{\boldsymbol\nu:\,\nu_r=0,\,\forall\,r\notin A\atop \sum_r r\nu_r=N}\tfrac{N!}{\prod\limits_r r^{\nu_r}
\nu_r!}\prod_{r} (x^{r}-1)^{\nu_{r}}\\
=[y^N]\ e^{-F(y)}\sum_{\la_1\in [N]}\binom{N-1}{\la_1-1}^{-1}\!\!(-1)^{\la^1} [x^{\la_1}]\tfrac{x}{1-x} e^{F(xy)}.\\
\end{multline}
Here
\begin{align*}
\binom{N-1}{\la_1-1}^{-1}&=\tfrac{(\la_1-1)!(N-\la_1)!}{(N-1)!}=N\int_0^1z^{\la_1-1}(1-z)^{N-\la_1}\,dz,\\
(-1)^{\la^1} [x^{\la_1}]\tfrac{x}{1-x} e^{F(xy)}&=(-1)^{N+1-\la_1} [x^{\la_1-1}]\tfrac{e^{F(xy)}}{1-x}\\
&=(-1)^N (-1)^{\la_1-1}[x^{\la_1-1}]\tfrac{e^{F(xy)}}{1-x}=(-1)^N [x^{\la_1-1}]\tfrac{\exp(F(-xy))}{1+x}.
\end{align*}
Therefore the sum in the last row of  \eqref{7} equals
\begin{multline*}
(-1)^N N \sum_{\la_1\in [N]} \biggl(\int_0^1z^{\la_1-1}(1-z)^{N-\la_1}\,dz\biggr)  [x^{\la_1-1}]\tfrac{\exp(F(-xy))}{1+x}\\
=(-1)^N N \int_0^1\sum_{\la_1\in [N]}z^{\la_1-1} [x^{\la_1-1}]\tfrac{\exp(F(-xy))}{1+x} (1-z)^{(N-1)-(\la_1-1)}\,dz\\
=(-1)^N N \int_0^1\sum_{\la_1\in [N]}[x^{\la_1-1}]\tfrac{\exp(F(-xyz))}{1+xz}\times \bigl[x^{(N-1)-(\la_1-1)}\bigr]\tfrac{1}{1-x(1-z)}\,dz\\
=(-1)^N N \int_0^1[x^{N-1}]\tfrac{\exp(F(-xyz))}{(1+xz)(1-x(1-z))}\,dz\\
=-N \int_0^1[x^{N-1}]\tfrac{\exp(F(xyz))}{(1-xz)(1+x(1-z))}\,dz.
\end{multline*}
In combination with \eqref{7} we have proved
\begin{equation}\label{8}
q_N(A)=-N[y^N]\, e^{-F(y)}\int_0^1[x^{N-1}]\,\tfrac{e^{F_A(xyz)}}{(1-xz)(1+x(1-z))}\,dz.
\end{equation}

{\bf (a)\/} Let us compute the integral on the right-hand side of \eqref{8} for a general $A$. Since
\begin{multline}\label{9}
\tfrac{1}{(1-xz)(1+x(1-z))}=\tfrac{1}{x}\tfrac{1}{1-xz}-\tfrac{1}{x}\cdot\tfrac{1}{1+x(1-z)}\\
=\tfrac{1}{x}\sum_{j\ge 0}(xz)^j-\tfrac{1}{x}\sum_{j\ge 0} (-1)^j x^j(1-z)^j,
\end{multline}
we have
\begin{multline*}
[x^{N-1}]\,\tfrac{\exp(F_A(xyz)}{(1-xz)(1+x(1-z))}=[x^N] e^{F_A(xyz)}\biggl(\sum_{j\ge 0}(xz)^j-\sum_{j\ge 0}(-1)^jx^j(1-z)^j\biggr)\\
=\sum_{k+j=N} (yz)^k [u^k] e^{F_A(u)} \bigl(z^j-(-1)^j(1-z)^j\bigr).\\
\end{multline*}
So, the integral on the right-hand side of \eqref{8}  equals
\begin{multline*}
\sum_{k+j=N}y^k [u^k]e^{F_A(u)}\int_0^1z^N\,dz-\sum_{k+j=N}(-1)^jy^k [u^k]e^{F_A(u)}\int_0^1z^k(1-z)^j\,dz\\
=\tfrac{1}{N+1}\sum_{k+j=N}y^k [u^k]e^{F_A(u)}-\tfrac{1}{N+1}\sum_{k+j=N}(-1)^j\binom{N}{j}^{-1}y^k [u^k]e^{F_A(u)}.
\end{multline*}
Then, since  $e^{-F_A(y)}=\sum_{\ell} y^{\ell} [v^{\ell}] e^{-F_A(v)}$, 
\begin{multline*}
[y^N]\, e^{-F_A(y)}\int_0^1[x^{N-1}]\,\tfrac{e^{F_A(xyz)}}{(1-xz)(1+x(1-z))}\,dz
=\tfrac{1}{N+1}\sum_{k+\ell=N} [v^{\ell}]e^{-F_A(v)} [u^k] e^{F_A(u)}\\-
\tfrac{1}{N+1}\sum_{k+j=N  \atop k+\ell=N}(-1)^j\binom{N}{j}^{-1} [v^{\ell}]e^{-F_A(v)}[u^k] e^{F_A(u)}\\
=\tfrac{1}{N+1}\sum_{k+\ell=N} [v^{\ell}]e^{-F_A(v)} [u^k] e^{F_A(u)}\\-
\tfrac{1}{N+1}\sum_{k+\ell=N}(-1)^j\binom{N}{\ell}^{-1} [v^{\ell}]e^{-F_A(v)}[u^k] e^{F_A(u)}.
\end{multline*}
The first sum equals $[x^N] e^{-F_A(x)+F_A(x)}=0$.  In the second sum, $[u^k]e^{F_A(u)}=p_k(A)$ for 
$k\ge 1$, and $[u^0]e^{F_A(u)}=1=:p_0(A)$.  Furthermore, recalling that  $A^c$ denotes the complement of $A$ in $\{1,2,\dots\}$, we get that

\[
[v^{\ell}] e^{-F_A(v)}=[v^{\ell}]\exp\biggl(-\log\tfrac{1}{1-v}+\sum_{r\in A^c}\tfrac{v^r}{r}\biggr).
\]
So 
 recalling that $p_0(A^c):=1$ and $p_{-1}(A^c):=0$, we have 
\begin{equation*} [v^{\ell}]e^{-F_A(v)}
=[v^{\ell}](1-v)e^{F_{A^c}(v)}=p_{\ell}(A^c)-p_{\ell-1}(A^c),\quad \ell\ge 0.
\end{equation*}
Collecting the pieces, we obtain
\begin{equation}\label{!}
q_N(A)=\tfrac{N}{N+1}\sum_{k+\ell=N}(-1)^{\ell}\binom{N}{\ell}^{-1} 
p_k(A)\bigl(p_{\ell}(A^c)-p_{\ell-1}(A^c)\bigr).
\end{equation}
Thus to evaluate the probability that the product of two cycles is an $A$-permutation, it suffices to know the probabilities that the uniformly distributed permutations of various length are $A$-permutations or 
$A^c$-permutations.
Explicitly,
\begin{align*}
p_k(A)&:=[x^k] \exp(F_A(x)]=[x^k]\exp\biggl(\sum_{r>1}\tfrac{x^r}{r}-\sum_{r<m}\tfrac{x^r}{r}\biggr)\\
&=[x^k]\tfrac{1}{1-x}\exp\biggl(-\sum_{r<m}\tfrac{x^r}{r}\biggr)=\sum_{j\ge 0}[x^{k-j}]\exp\biggl(-\sum_{r<m}\tfrac{x^r}{r}\biggr)\\
&=\sum_{\mu\ge 0}\tfrac{(-1)^{\mu}}{\mu!}\sum_{j\ge 0}[x^{k-j}]\biggl(\sum_{r<m}\tfrac{x^r}{r}\biggr)^{\mu}\\
&=\sum_{\mu\ge 0}(-1)^{\mu}\sum_{j\le k}\sum_{\mu_1,\dots,\mu_{m-1}\atop \sum_r\mu_r=\mu,\,
\sum_r r\mu_r=k-j}\prod_{r=1}^{m-1}\tfrac{1}{r^{\mu_r}\mu_r!}\\
&=\sum_{\mu_r\ge 0,\,r\in A^c\atop
\sum_r r\mu_r\le k}\prod_{r\in A^c}\tfrac{(-1)^{\mu_r}}{r^{\mu_r}\mu_r!},
\end{align*}
and 

\begin{eqnarray*}
p_{\ell}(A^c) &= & [x^{\ell}]\exp\biggl(\sum_{r\in A^c}\tfrac{x^r}{r}\biggr)=\sum_j [x^{\ell}]\tfrac{\biggl(\sum_{r\in A^c}\tfrac{x^r}{r}\biggr)^j}{j!}\\
& = & \sum_{\sum_{r\in A^c} r \mu_r=\ell}\,\prod_{r\in A^c}\tfrac{1}{r^{\mu_r}\mu_r!}.
\end{eqnarray*}

{\bf (b)\/} For $A=\mathcal E$ and $A=\mathcal O$, we get explicit formulas for $q_N(A)$
by combining \eqref{!} with \eqref{2**} for those $A$'s. However it turns out possible to obtain considerably simpler, sum-type, formulas 
for $q_N(\mathcal E)$ and $q_N(\mathcal O)$ with only positive summands. It is done  by properly changing the order of steps in the derivation above.

Start with $q_N(\mathcal E)$. In this case, we carry the operation
$[y^N]$ inside the integration in \eqref{8}:
\begin{equation}\label{1e}
\begin{aligned}
q_N(\mathcal E)&=-N\int_0^1[x^{N-1}y^N]\tfrac{(1-y^2)^{1/2}}{(1-(xyz)^2)^{1/2}}\cdot\tfrac{1}{(1-xz)(1+x(1-z))}\,dz\\
&=-N\int_0^1[x^{N-1}]H_N(xz)\cdot \tfrac{1}{(1-xz)(1+x(1-z))}\,dz,\\
\end{aligned}
\end{equation}
where 
\begin{multline}\label{1.5e}
H_N(xz):=[y^N]\tfrac{(1-y^2)^{1/2}}{(1-(xyz)^2)^{1/2}}\\
=[y^N]\biggl(1+\sum_{k\ge 1}\binom{1/2}{k}(-y^2)^k\biggr)\cdot\biggl(\sum_{\ell\ge 0}\binom{-1/2}{\ell}(-(xyz)^2)^{\ell}\biggr)\\
=[y^N]\biggr(1-\tfrac{1}{2}\sum_{k\ge 1}\tfrac{y^{2k}}{k!}\prod_{r=1}^{k-1}(r-1/2)\biggr)\cdot\biggl(\sum_{\ell\ge 0}\tfrac{(xyz)^{2\ell}}{\ell!}\prod_{r'=0}^{\ell-1}(r'+1/2)\biggr)\\
=\tfrac{(xz)^N}{(N/2)!}\prod_{r'=0}^{N/2-1}(r'+1/2)-\tfrac{1}{2}\sum_{k\ge 1 \atop 2k+2\ell=N}\tfrac{(xz)^{2\ell}}{k!\ell!}\prod_{r=1}^{k-1}(r-1/2)\prod_{r'=0}^{\ell-1}(r'+1/2).
\end{multline}
So, combining \eqref{1.5e} and \eqref{9}, we evaluate
\begin{multline*}
[x^N] H_N(xz)\sum_{j\ge 0} x^j\bigl[z^j-(-1)^j(1-z)^j\bigr]\\
=-\tfrac{1}{2}\sum_{k\ge 1, 2k+2\ell=N,\atop j+2\ell=N}\tfrac{z^{2\ell}}{k!\ell!}\bigl[z^j-(-1)^j(1-z)^j\bigr]
\prod_{r=1}^{k-1}(r-1/d)\prod_{r'=0}^{\ell-1}(r'+1/d)\\
=-\tfrac{1}{2}\sum_{k\ge 1,\, 2k+2\ell=N}\tfrac{1}{k!\ell!}\,\bigl[z^N-(-1)^{2k}\,(1-z)^{2k}z^{2\ell}\bigr]
\prod_{r=1}^{k-1}(r-1/2)\prod_{r'=0}^{\ell-1}(r'+1/2).
\\
\end{multline*}
Integrating over $z\in [0,1]$, using 
\[
\int_0^{1}(1-z)^{2k}z^{2\ell}\,dz=\tfrac{1}{N+1}\binom{N}{2\ell}^{-1},
\]
 and multiplying by $-N$, we obtain
\begin{align*}
q_N(\mathcal E)&=\tfrac{N}{2(N+1)}\sum_{k\ge 1,\, 2k+2\ell=N}\tfrac{1}{k!\ell!}\biggl(1-(-1)^{2k}\binom{N}{2k}^{-1}\biggr)\\
&\qquad\times \prod_{r=1}^{k-1}(r-1/2)\prod_{r'=0}^{\ell-1}(r'+1/2).
\\
\end{align*}
Since
\[
\tfrac{1}{k!} \prod_{r=1}^{k-1}(r-1/2)=\tfrac{1}{2^{2(k-1)}k}\binom{2(k-1)}{k-1},\quad 
\tfrac{1}{\ell!}\prod_{r'=0}^{\ell-1}(r'+1/2)=\tfrac{1}{2^{2\ell}}\binom{2\ell}{\ell},
\]
we conclude that
\[
q_N(\mathcal E)=\tfrac{2N}{(N+1) 2^N}\sum_{k+\ell=N/2}\tfrac{1}{k}\binom{2(k-1)}{k-1}\binom{2\ell}{\ell}\biggl(1-\binom{N}{2k}^{-1}\biggr)\!,\quad N\text{ even}.
\]

Consider $A=\mathcal O$. This time,
\begin{align*}
\exp\bigl(F_{\mathcal O}(xyz)-F_{\mathcal O}(y)\bigr)&=\sqrt{\tfrac{1+xyz}{1-xyz}}\cdot \sqrt{\tfrac{1-y}{1+y}}\\
&=\tfrac{1+xyz}{\sqrt{1-(xyz)^2}}\cdot\tfrac{1-y}{\sqrt{1-y^2}},
\end{align*}
and the equation \eqref{8} yields
\begin{align*}
q_N(\mathcal O)&=-[x^{N-1}y^N]N\int_0^1\tfrac{1+xyz}{\sqrt{1-(xyz)^2}}\cdot\tfrac{1-y}{\sqrt{1-y^2}}\cdot\tfrac{1}{(1-xz)(1+x(1-z))}\,dz\\
&=-[x^{N-1}]N\int_0^1H_N(xz)\cdot \tfrac{1}{(1-xz)(1+x(1-z))}\,dz,\\
\end{align*}
where 
\begin{multline}\label{9.5}
H_N(xz):=[y^N]\tfrac{1+xyz}{\sqrt{1-(xyz)^2}}\cdot\tfrac{1-y}{\sqrt{1-y^2}}\\
=[y^N] (1+xyz)(1-y)\sum_{\ell\ge 0}\binom{2\ell}{\ell}\bigl(\tfrac{(xyz)^2}{4}\bigr)^\ell\cdot\sum_{k\ge 0}\binom{2k}{k}\bigl(\tfrac{y^2}{4}\bigr)^k.
\end{multline}
{\bf (b1)\/} If  $N$ is even, then \eqref{9.5} yields
\begin{multline*}
H_N(xz)=[y^N](1-xy^2z)\sum_{\ell\ge 0}\binom{2\ell}{\ell}\bigl(\tfrac{(xyz)^2}{4}\bigr)^\ell\cdot\sum_{k\ge 0}\binom{2k}{k}\bigl(\tfrac{y^2}{4}\bigr)^k\\
=2^{-N}\sum_{2k+2\ell=N}\binom{2k}{k}\binom{2\ell}{\ell}(xz)^{2\ell}-2^{-N+2}\sum_{2k+2\ell=N-2}\binom{2k}{k}\binom{2\ell}{\ell}(xz)^{2\ell+1}.
\end{multline*}
Therefore
\begin{multline*}
[x^N] H_N(xz)\sum_{j\ge 0}(xz)^j=2^{-N}z^N\sum_{2k+2\ell=N}\binom{2k}{k}\binom{2\ell}{\ell}\\
-2^{-N+2}z^N\sum_{2k+2\ell=N-2}\binom{2k}{k}\binom{2\ell}{\ell},\\
\end{multline*}
and
\begin{multline*}
[x^N] H_N(xz)\sum_{j\ge 0}(-1)^jx^j(1-z)^j=2^{-N}\sum_{2k+2\ell=N}\binom{2k}{k}\binom{2\ell}{\ell}(1-z)^{2k}z^{2\ell}\\
+2^{-N+2}\sum_{2k+2\ell=N-2}\binom{2k}{k}\binom{2\ell}{\ell}(1-z)^{2k+1}z^{2\ell+1}.\\
\end{multline*}
Therefore
\begin{align*}
&q_N(\text{odd})=-N\int_0^1[x^{N-1}] \Bigl(H_N(xz)\cdot\tfrac{1}{(1-xz)(1+x(1-z))}\Bigr)\,dz\\
& =-N\int_0^1 \biggl[2^{-N}z^N\sum_{2k+2\ell=N}\binom{2k}{k}\binom{2\ell}{\ell}-2^{-N+2}z^N\sum_{2k+2\ell=N-2}\binom{2k}{k}\binom{2\ell}{\ell}\\
&\qquad\qquad\qquad\qquad\qquad-2^{-N}\sum_{2k+2\ell=N}\binom{2k}{k}\binom{2\ell}{\ell}(1-z)^{2k}z^{2\ell}\\
&\qquad\qquad\qquad\qquad\qquad\qquad-2^{-N+2}\sum_{2k+2\ell=N-2}\binom{2k}{k}\binom{2\ell}{\ell}(1-z)^{2k+1}z^{2\ell+1}\biggr]\,dz.\\
\end{align*}
Integrating over $z\in [0,1]$, we obtain 
\begin{align*}
&q_N(\mathcal O)=-\tfrac{N}{2^N(N+1)}\sum_{2k+2\ell=N}\binom{2k}{k}\binom{2\ell}{\ell}+\tfrac{N}{2^{N-2}(N+1)}
\sum_{2k+2\ell=N-2}\binom{2k}{k}\binom{2\ell}{\ell}\\
&\qquad\qquad\qquad\qquad\qquad+\tfrac{N}{2^N(N+1)}\sum_{2k+2\ell=N}\binom{2k}{k}\binom{2\ell}{\ell}
\binom{N}{2\ell}^{-1}\\
&\qquad\qquad\qquad\qquad\qquad\qquad+\tfrac{N}{2^{N-2}(N+1)}\sum_{2k+2\ell=N-2}\binom{2k}{k}\binom{2\ell}{\ell}
\binom{N}{2\ell+1}^{-1}.
\end{align*}
Here
\begin{eqnarray*}
\sum_{2k+2\ell=N}\binom{2k}{k}\binom{2\ell}{\ell}=[\eta^{N/2}]\biggl(\sum_{k\ge 0}\binom{2k}{k}\eta^k\biggr)^2\\
=[\eta^{N/2}]\tfrac{1}{1-4\eta}=4^{N/2}=2^N,
\end{eqnarray*}
and likewise
\[
\sum_{2k+2\ell=N-2}\binom{2k}{k}\binom{2\ell}{\ell}=2^{N-2}.
\]
Therefore the formula for $q_N(\text{odd})$ simplifies to 
\begin{multline}\label{11}
q_N(\mathcal O)=\tfrac{N}{2^N(N+1)}\sum_{2k+2\ell=N}\binom{2k}{k}\binom{2\ell}{\ell}
\binom{N}{2\ell}^{-1}\\
+\tfrac{N}{2^{N-2}(N+1)}\sum_{2k+2\ell=N-2}\binom{2k}{k}\binom{2\ell}{\ell}
\binom{N}{2\ell+1}^{-1}.
\end{multline}

{\bf (b2)\/} If $N$ is odd, then \eqref{9.5} yields
\begin{multline*}
H_N(xz)=[y^N](y(xz-1))\sum_{\ell\ge 0}\binom{2\ell}{\ell}\bigl(\tfrac{(xyz)^2}{4}\bigr)^\ell\cdot\sum_{k\ge 0}\binom{2k}{k}\bigl(\tfrac{y^2}{4}\bigr)^k\\
=\tfrac{xz-1}{2^{N-1}}\sum_{2\ell+2k=N-1}\binom{2\ell}{\ell}\binom{2k}{k}(xz)^{2\ell}.
\end{multline*}
Consequently
\begin{multline*}
[x^N]H_N(xz)\sum_{j\ge 0}(xz)^j=\tfrac{1}{2^{N-1}}\biggl[z [x^{N-1}]\sum_{2\ell+2k=N-1}\binom{2\ell}{\ell}\binom{2k}{k}(xz)^{2\ell}\sum_{j\ge 0}(xz)^j\\
-[x^N]\sum_{2\ell+2k=N-1}\binom{2\ell}{\ell}\binom{2k}{k}(xz)^{2\ell}\sum_{j\ge 0}(xz)^j\biggr]\\
=\tfrac{1}{2^{N-1}}\biggl[z\sum_{2\ell+2k=N-1}\binom{2\ell}{\ell}\binom{2k}{k} z^{2\ell+2k}-z^N\sum_{2\ell+2k=N-1}\binom{2\ell}{\ell}\binom{2k}{k}\biggr]=0.
\end{multline*}
And similarly
\begin{multline*}
[x^N]H_N(xz)\sum_{j\ge 0}(-1)^jx^j(1-z)^j\\
=\tfrac{1}{2^{N-1}}\biggl[\sum_{2\ell+2k=N-1}\binom{2\ell}{\ell} \binom{2k}{k} z^{2\ell+1}(1-z)^{2k}\\
+\sum_{2\ell+2k=N-1}\binom{2\ell}{\ell} \binom{2k}{k} z^{2\ell}(1-z)^{2k+1}\biggr].
\end{multline*}
So, integrating $-[x^N]H_N(xz)\sum_{j\ge 0}(-1)^jx^j(1-z)^j$ over $z\in [0,1]$ and multiplying the result by $-N$, we obtain
\begin{equation}\label{12}
q_N(\mathcal O)\!=\!\tfrac{N}{(N+1)2^{N-1}}\!\!\sum_{2\ell+2k=N-1}\!\!\binom{2\ell}{\ell}\binom{2k}{k}\biggl(\!\binom{N}{2\ell}^{-1}\!\! +\binom{N}{2\ell+1}^{-1}\biggr).
\end{equation}
The proof of the part {\bf (b)\/} is complete, and so is the proof of Theorem \ref{thm1*}.
\bi
{\bf Remarks.\/} {\bf(1)\/} The probabilities $q_N(A)$ do not change if instead of considering all pairs $(p,q)$ of cycles of length $N$, selected from 
$\mathcal C_N\times \mathcal C_N$ uniformly at random, we set
$p=(12\cdots N)$ and we select $q$ from $C_N$ uniformly at random.

\section{Block Transposition Sorting} \label{btsorting}  
 Our motivation for obtaining exact and asymptotic results regarding the product of two long cycles comes from
a biologically motivated sorting algorithms called {\em block transposition sorting}, Bafna and Pevzner \cite{bafna}. A step in this sorting algorithm consists of swapping {\em adjacent} blocks  of entries.  The lengths of the two blocks do not have to be the same.
That is, 
\[u=u_1u_2\cdots u_iu_{i+1}\cdots u_ju_{j+1} \cdots u_ku_{k+1}\cdots u_N\]
 is turned into the permutation
\[v=u_1u_2\cdots u_iu_{j+1}\cdots u_ku_{i+1} \cdots u_ju_{k+1}\cdots u_N.\]
If $u,v\in S_N$, and $t$ is the smallest  integer so that there exists a sequence of
$t$ block transpositions that turn $u$ into $v$, then we say that
$t=\mathrm{btd}(u,v)$ is the {\em block transposition distance} of  $u$ and $v$.

The block transposition
distance of two permutations is {\em left-invariant}, that is, if $u$, $v$ and $r$ are permutations of length $N$, then
$\mathrm{btd}(u,v)=\mathrm{btd}(ru,rv)$. Selecting $r=v^{-1}$, this means that $\mathrm{btd}(u,v)=\mathrm{btd}(v^{-1}u,\id)$, where $\mathrm{id}$ denotes the identity permutation.
This reduces our distance-measuring problem to a {\em sorting problem}, that is, we can focus, without loss of generality, on the block transposition distance of a given permutation
from the {\em identity permutation}.  Therefore,  we will write $\mathrm{btd}(u)$ for $\mathrm{btd}(u,\id)$. 

Several natural and difficult questions arise. Which permutation $u$ is the furthest away from $\id$? What is $\mathrm{td}(N)=\max_{u\in S_N}\mathrm{btd}(u)$?
 It is known, Eriksson et al., \cite{eriksson}, that $\mathrm{td}(N) \leq \lfloor 2N/3 \rfloor$. On the lower end \cite{eriksson},  
$\mathrm{btd}(N\cdots 21) = \lceil(N+1)/2 \rceil $, which led to the conjecture $\mathrm(td)(N)=\lceil(N+1)/2 \rceil $. However, this has been disproved, Elias and Hartman \cite{elias}, for $N=4k+1\geq 17$.
For such $N$, there is a permutation $u\in S_N$ so that $\mathrm{btd}(u)=2k+2=\lceil (N+1)/2 \rceil +1$. So there is a considerable gap between the best known lower and upper bound for 
$\mathrm{td}(N)$. 

As there are almost no known permutations $u$ satisfying the inequality $\mathrm{btd(u)}> (N+1)/2$, there is intrinsic interest in finding out how many permutations $v$ 
satisfy $\mathrm{btd}(v)=\lceil (N+1)/2 \rceil$. In this paper, we will find a lower bound for the number of such permutations. This lower bound will be significantly higher than  what could be
deduced before, based on the results in Cranston et al. \cite{cranston} and Christie \cite{christie}. 

Crucially, there is a bijection $C$ from the set of all permutations of length $N$ to the set of all 
permutations of length $N+1$ that are of the form $s \cdot z$, where $s=(12\cdots (N+1))$ and $\nu(z)=1$, that is, $z$  is a
cycle of length $N+1$. This bijection is defined \cite{christie} with the help of a graph called the {\em cycle graph}, but we will not need that definition in this paper.  
Let $\nu_{odd}(\pi)$ be the number of odd cycles of the
permutation $\pi$. The following lemma shows the importance of the bijection $C$ for us. 

\begin{lemma} \label{bafna}  \cite{bafna}
For all permutations $u$ of length $N$, the inequality
$\mathrm{btd}(u) \geq \frac{N+1-\nu_{odd}(C(u))}{2} $ holds. 
\end{lemma}

Clearly, the upper bound of Lemma \ref{bafna} is the strongest when $\nu_{odd}(C(u))=0$, that is, when 
all cycles of $C(u)$ are even. We are now in a position to prove a lower bound for the number of  such permutations $u$. 

 \begin{theorem}  \label{lowerbound} The number of permutations of length $n$ that take at least $\lceil (N+1)/2 \rceil$ block transpositions to sort is at least $N! \,q_{N+1}(\mathcal E)$, where $q_{\cdot}(\mathcal E)$ is defined by the first equation in $(\ref{5*})$.

Consequently, the fraction of such permutations among all permutations of length $n$ is asymptotically at least $1/\sqrt{\pi N/2}$. 
\end{theorem}

\begin{proof} Lemma \ref{bafna} shows that $\mathrm{btd}(u)\geq \lceil (N+1)/2 \rceil$ when all cycles of $C(u)$ are even. The existence of bijection $C$ shows that the number of such permutations 
equals the number of pairs $q$ that consist of a single cycle of length $N+1$ so that
$(12 \cdots (N+1) ) \cdot q$ has even cycles only. And, by the definition of $q_{N+1}(\mathcal E)$,  the number of such pairs is $N!\, q_{N+1}(\mathcal E)$, and their portion
 among all $N!$ permutations is asymptotically equal to $(\pi N/2)^{-1/2}$,
see Theorem \ref{thm1*}.
\end{proof}

Note that Theorem \ref{lowerbound} is significant for two reasons.

First,  the result of Theorem \ref{lowerbound} is significantly stronger than what was known before. In particular,  it follows from Lemma \ref{bafna} that   $\mathrm{btd}(u)\geq \lfloor N/2 \rfloor$ holds when
 $C(u)$ has exactly one odd cycle. The number of such permutations $u$ is known, Boccara \cite{Boc}, B\'ona and Pittel \cite{ bona-pittel}, to be  $2N!/(N+2)$ if $N$ is even, and 0 if $N$ is odd,  so the ratio of such permutations among all  permutations of length $N$ is $2/(N+2)$ if $n$ is even. This lower bound is weaker than what we have just proved, by a factor of 
$c/\sqrt{N}$. Another lower bound, also significantly weaker than the one we just proved, can be deduced from the results in \cite{cranston}, 
where the authors considered {\em cut-and-paste} sorting, an operation that is stronger than block transposition sorting.  
However, that lower bound just shows that the relevant ratio is at least $\sqrt{N}/2^N$, which is again weaker than the result above. 

Second, and probably more importantly,  we have just proved the existence of a lot of permutations that are very close to the known worst case. The decreasing permutation 
$N\cdots 21$ takes $ \lceil(N+1)/2 \rceil $ block transpositions to sort, and as we mentioned before, no permutation of length $N$ is known to be harder to sort unless $N=4k+1\geq 17$, 
and even in that case, only one permutation is known to be harder to sort, and that permutation needs only one additional block transposition to be sorted \cite{elias}. 
 So, Theorem 
\ref{lowerbound} shows that at least $1/\sqrt{\pi N/2}$ of all permutations are either equal to the strongest
known construction for worst case, or are just one step away from it. 

\section{Further directions} 
It is natural to ask if other special cases Theorem \ref{thm1*} have application in sorting algorithms. The answer is in the affirmative.  A {\em block interchange} is like a block transposition except that the blocks that are swapped do not have to be consecutive.  Then a theorem in \cite{christie} shows that $(N+1-\nu(C(u))/2$ is {\em equal}
 to the
block interchange distance of the permutation $u$ of length $n$ to the identity.  It follows that if 
 $C(u)$ has at least one even cycle, then the block interchange distance and the block transposition distance of $u$ to the identity must be different. We hope to explore this fact in a subsequent  paper.

{\bf Appendix 1.\/} The argument mimicks the proof the formula for $q_N(\mathcal E)$, i.e. the case $d=2$. This time $N$ is divisible by $d$.
The equation \eqref{1e} holds for $q_N(\mathcal D)$, where
\begin{multline}\label{a1}
H_N(xz):=[y^N]\tfrac{(1-y^d)^{1/d}}{(1-(xyz)^d)^{1/d}}\\
=[y^N]\biggl(1+\sum_{k\ge 1}\binom{1/d}{k}(-y^d)^k\biggr)\cdot\biggl(\sum_{\ell\ge 0}\binom{-1/d}{\ell}(-(xyz)^d)^{\ell}\biggr)\\
=[y^N]\biggr(1-\tfrac{1}{d}\sum_{k\ge 1}\tfrac{y^{dk}}{k!}\prod_{r=1}^{k-1}(r-1/d)\biggr)\cdot\biggl(\sum_{\ell\ge 0}\tfrac{(xyz)^{d\ell}}{\ell!}\prod_{r'=0}^{\ell-1}(r'+1/d)\biggr)\\
=\tfrac{(xz)^N}{(N/d)!}\prod_{r'=0}^{N/d-1}(r'+1/d)-\tfrac{1}{d}\sum_{k\ge 1,\,dk+d\ell=N}\tfrac{(xz)^{d\ell}}{k!\ell!}\prod_{r=1}^{k-1}(r-1/d)\prod_{r'=0}^{\ell-1}(r'+1/d).
\end{multline}
So, according to the equation \eqref{1e} for $q_N(\mathcal D)$, and \eqref{9}, we evaluate
\begin{multline*}
[x^N] H_N(xz)\sum_{j\ge 0} x^j\bigl[z^j-(-1)^j(1-z)^j\bigr]\\
=-\tfrac{1}{d}\sum_{k\ge 1, dk+d\ell=N,\atop j+d\ell=N}\tfrac{z^{d\ell}}{k!\ell!}\bigl[z^j-(-1)^j(1-z)^j\bigr]
\prod_{r=1}^{k-1}(r-1/d)\prod_{r'=0}^{\ell-1}(r'+1/d)\\
=-\tfrac{1}{d}\sum_{k\ge 1,\, dk+d\ell=N}\tfrac{1}{k!\ell!}\,\bigl[z^N-(-1)^{dk}\,(1-z)^{dk}z^{d\ell}\bigr]
\prod_{r=1}^{k-1}(r-1/d)\prod_{r'=0}^{\ell-1}(r'+1/d).
\\
\end{multline*}
Integrating over $z\in [0,1]$ and multiplying by $-N$, we obtain
\begin{align*}
q_N(\mathcal D)&=\tfrac{N}{d(N+1)}\sum_{k\ge 1,\, dk+d\ell=N}\tfrac{1}{k!\ell!}\biggl(1-(-1)^{dk}\binom{N}{d\ell}^{-1}\biggr)\\
&\qquad\times \prod_{r=1}^{k-1}(r-1/d)\prod_{r'=0}^{\ell-1}(r'+1/d).
\\
\end{align*}
{\bf Appendix 2.\/} Proof of \eqref{new10*}. Using \eqref{4*} and \eqref{star}, we obtain
\begin{align*}
(N-1)!\,q_N(A)&=\tfrac{1}{N+1}\sum_{k=0}^N(-1)^{N-k-1}(N-k-1)k!\sum_{j\le k}\tfrac{(-1)^j}{j!}\\
&=\tfrac{1}{N+1}\sum_{j=0}^N\tfrac{(-1)^j}{j!}\biggl[\sum_{j\le k\le N}(-1)^{N-k-1} (N-k-1)\,k!\biggr].
\end{align*}
Here
\begin{multline*}
\sum_{j\le k\le N}(-1)^{N-k-1} (N-k-1)\,k!\\
=N (-1)^{N-1}\sum_{j\le k\le N}(-1)^k k! -(-1)^N\sum_{j+1\le k\le N+1}(-1)^k k!\\
=(-1)^{N-1}\biggl[(N+1)\sum_{j\le k\le N}(-1)^k k!-(-1)^j\,j!+(-1)^{N+1}\,(N+1)!\biggr].
\end{multline*}
Therefore
\begin{multline*}
(N-1)!\,q_N(A)=\tfrac{(-1)^{N-1}}{N+1}\sum_{j=0}^N\tfrac{(-1)^j}{j!}\\
\times\biggl[(N+1)\sum_{j\le k\le N}(-1)^k k!-(-1)^j\,j!+(-1)^{N+1}\,(N+1)!\biggr]\\
=(-1)^{N-1}\biggl[\sum_{j=0}^N\tfrac{(-1)^j}{j!}\sum_{j\le k\le N}(-1)^k\,k!-1\biggr]+N!\sum_{j=0}^N\tfrac{(-1)^j}{j!}\\
=(-1)^{N-1}\biggl[\sum_{j=0}^{N-1}\tfrac{(-1)^j}{j!}\sum_{j\le k\le N}(-1)^k\,k!\biggr]+N!\sum_{j=0}^N\tfrac{(-1)^j}{j!}\\
=(-1)^{N-1}\biggl[\sum_{j=0}^{N-1}\tfrac{(-1)^j}{j!}\sum_{j\le k\le N-1}(-1)^k\,k!\biggr]
-N!\sum_{j=0}^{N-1}\tfrac{(-1)^j}{j!}+N!\sum_{j=0}^N\tfrac{(-1)^j}{j!}\\
=(-1)^{N-1}\biggl[\sum_{j=0}^{N-1}\tfrac{(-1)^j}{j!}\sum_{j\le k\le N-1}(-1)^k\,k!-1\biggr]\\
=(-1)^{N-1}\sum_{j=0}^{N-2}\tfrac{(-1)^j}{j!}\sum_{j\le k\le N-1}(-1)^k\,k!.
\end{multline*}
\qed
\end{document}